\documentclass[12pt,a4paper,draft]{article}
\usepackage{amsfonts,amssymb}
\usepackage{theorem}

{\theorembodyfont{\rm}
  \newtheorem{defi}{Definition}[section]
  \newtheorem{rem}[defi]{Remark}
  
  \newtheorem{exas}[defi]{Examples}}
  \newtheorem{lem}[defi]{Lemma}
  
  \newtheorem{thm}[defi]{Theorem}
  
\newcommand{\CC}{{\mathbb C}}

\newcommand{\PP}{{\mathbb P}}

\newcommand{\ZZ}{{\mathbb Z}}

\newcommand{\cG}{{\mathcal G}}

\newcommand{\Aut}{{\mathrm{Aut}}}

\newcommand{\End}{{\mathrm{End}}}

\newcommand{\Char}{{\mathrm{char}}\,}

\newcommand{\id}{{\mathrm{id}}}
\newcommand{\dis}
                 {{\mathrel{\scriptstyle{\triangle}}}}

\newcommand{\GL}{{\mathrm{GL}}}
\newcommand{\GE}{{\mathrm{GE}}}
\newcommand{\E}{{\mathrm{E}}}

\newcommand{\dist}{{\mathrm{dist}}}
\newcommand{\diag}{{\mathrm{diag}}}
\let\phi=\varphi

\let\theta=\vartheta

\newcommand{\DelimArray}[4]{\left#1\begin{array}{*{#3}{c}}#4\end{array}\right#2}
\newcommand{\SDelimArray}[4]{\hbox{\scriptsize\arraycolsep=.5\arraycolsep
  $\left#1\!\!\begin{array}{*{#3}{c}}#4\end{array}\!\!\right#2$}}

\newcommand{\Mat}{\DelimArray()}
\newcommand{\SMat}{\SDelimArray()}

\newenvironment{proof}
    {\begin{trivlist} \item {\sl Proof:}} 
    {\/ $\square$ \end{trivlist}}

\begin{document}
\title{The Connected Components of the\\ Projective Line over a Ring}
\author{Andrea Blunck\thanks{Supported by a Lise Meitner
 Research Fellowship
of the Austrian Science Fund (FWF), project M574-MAT.} \and {Hans
Havlicek}}
\date{}
\maketitle

\begin{abstract}
\noindent The main result of the present paper is that the projective
line over a ring $R$ is connected with respect to the relation
``distant'' if, and only if, $R$ is a $\GE_2$-ring.

\end{abstract}
\parskip1mm
\parindent0cm
\section{Introduction}
One of the basic notions for the projective line $\PP(R)$ over a ring
$R$ is the relation {\em distant\/} ($\dis$) on the point set.
Non-distant points are also called {\em parallel}. This terminology
goes back to the projective line over the real dual numbers, where
parallel points represent parallel spears of the Euclidean plane
\cite[2.4]{benz-73}.

We say that $\PP(R)$ is {\em connected\/} (with respect to $\dis$) if
the following holds: For any two points $p$ and $q$ there is a finite
sequence of points starting at $p$ and ending at $q$ such that each
point other than $p$ is distant from its predecessor. Otherwise
$\PP(R)$ is said to be {\em disconnected}. For each {\em connected
component\/} a {\em distance function\/} and a {\em diameter\/} (with
respect to $\dis$) can be defined in a natural way.

One aim of the present paper is to characterize those rings $R$ for which
$\PP(R)$ is connected. Here we use certain subgroups of the group $\GL_2(R)$ of
invertible $2\times 2$-matrices over $R$, namely its {\em elementary
subgroup\/} $\E_2(R)$ and the subgroup $\GE_2(R)$ generated by $\E_2(R)$ and
the set of all invertible diagonal matrices. It turns out that $\PP(R)$ is
connected exactly if $R$ is a $\GE_2$-ring, i.e., if $\GE_2(R)=\GL_2(R)$.

Next we turn to the diameter of connected components. We show that
all connected components of $\PP(R)$ share a common diameter.

It is well known that $\PP(R)$ is connected with diameter $\leq 2$ if
$R$ is a ring of stable rank $2$. We give explicit examples of rings
$R$ such that $\PP(R)$ has one of the following properties: $\PP(R)$
is connected with diameter $3$, $\PP(R)$ is connected with diameter
$\infty$, and $\PP(R)$ is disconnected with diameter $\infty$. In
particular, we show that there are {\em chain geometries\/} over
disconnected projective lines.

\section{Preliminaries}\label{se:prelim}

Throughout this paper we shall only consider associative rings with a
unit element $1$, which is inherited by subrings and acts unitally on
modules. The trivial case $1=0$ is not excluded. The group of
invertible elements of a ring $R$ will be denoted by $R^*$.

Firstly, we turn to the projective line over a ring: Consider the
free left $R$-module $R^2$. Its automorphism group is the group
$\GL_2(R)$ of invertible $2\times 2$-matrices with entries in $R$. A
pair $(a,b)\in R^2$ is called {\em admissible}, if there exists a
matrix in $\GL_2(R)$ with $(a,b)$ being its first row. Following
\cite[785]{herz-95}, the {\em projective line over\/} $R$ is the
orbit of the free cyclic submodule $R(1,0)$ under the action of
$\GL_2(R)$. So
\begin{displaymath}
  \PP(R):=R(1,0)^{\GL_2(R)}
\end{displaymath}
or, in other words, $\PP(R)$ is the set of all $p\leq R^2$ such that
$p=R(a,b)$ for an admissible pair $(a,b)\in R^2$. As has been pointed
out in \cite[Proposition 2.1]{blu+h-00a}, in certain cases
$R(x,y)\in\PP(R)$ does not imply the admissibility of $(x,y)\in R^2$.
However, throughout this paper we adopt the convention that points
are represented by admissible pairs only. Two such pairs represent
the same point exactly if they are left-proportional by a unit in
$R$.

The point set $\PP(R)$ is endowed with the symmetric relation {\em
distant\/} ($\dis$) defined by
\begin{equation}\label{def-distant}
  \dis:=(R(1,0), R(0,1))^{\GL_2(R)}.
\end{equation}
Letting $p=R(a,b)$ and $q= R(c,d)$ gives then
\begin{displaymath}
  p\,\dis\, q\iff \Mat2{a&b\\c&d}\in \GL_2(R).
\end{displaymath}
In addition, $\dis$ is anti-reflexive exactly if $1\ne 0$.

The vertices of the {\em distant graph\/} on $\PP(R)$ are the points
of $\PP(R)$, the edges of this graph are the unordered pairs of
distant points. Therefore basic graph-theoretical concepts are at
hand: $\PP(R)$ can be decomposed into {\em connected components\/}
(maximal connected subsets), for each connected component there is a
{\em distance function\/} ($\dist(p,q)$ is the minimal number of
edges needed to go from vertex $p$ to vertex $q$), and each connected
component has a {\em diameter\/} (the supremum of all distances
between its points).

Secondly, we recall that the set of all {\em elementary matrices\/}
\begin{equation}\label{elmatrizen}
  B_{12}(t):=\Mat2{1&t\\0&1} \mbox{ and }
  B_{21}(t):=\Mat2{1&0\\t&1}
  \mbox{ with } t\in R
\end{equation}
generates the {\em elementary subgroup\/} $\E_2(R)$ of $\GL_2(R)$. The group
$\E_2(R)$ is also generated by the set of all matrices
\begin{equation}\label{erzeugende}
  E(t):=\Mat2{t&1\\-1&0}=B_{12}(1)
  \cdot B_{21}(-1)\cdot B_{12}(1)\cdot B_{21}(t)
  \mbox{ with } t\in R,
\end{equation}
since $B_{12}(t)=E(-t)\cdot E(0)^{-1}$ and $B_{21}(t)=E(0)^{-1}\cdot
E(t)$. Further, $E(t)^{-1}=E(0)\cdot E(-t)\cdot E(0)$ implies that
all finite products of matrices $E(t)$ already comprise the group
$\E_2(R)$.

The subgroup of $\GL_2(R)$ which is generated by $\E_2(R)$ and the set of all
invertible diagonal matrices is denoted by $\GE_2(R)$. By definition, a {\em
$\GE_2$-ring\/} is characterized by $\GL_2(R)=\GE_2(R)$; see, among others,
\cite[5]{cohn-66} or \cite[114]{silv-81}.

\section{Connected Components}

We aim at a description of the connected components of the projective
line $\PP(R)$ over a ring $R$. The following lemma, although more or
less trivial, will turn out useful:

\begin{lem}\label{le:matrizen}
Let $X'\in\GL_2(R)$ and suppose that the $2\times 2$-matrix $X$ over
$R$ has the same first row as $X'$.  Then $X$ is invertible if, and
only if, there is a matrix
\begin{equation}\label{su-matrix}
M = \Mat2{1&0\\s&u}\in \GE_2(R)
\end{equation}
such that $X = MX'$.
\end{lem}
\begin{proof}
Given $X'$ and $X$ then $XX'^{-1} = \SMat2{1&0\\s&u}=:M$ for some
$s,u\in R$. Further, $X=MX'$ is invertible exactly if $u\in R^*$.
This in turn is equivalent to (\ref{su-matrix}).
\end{proof}

Here is our main result, where we use the generating matrices of
$\E_2(R)$ introduced in (\ref{erzeugende}).

\begin{thm}\label{th:komponenten}
Denote by $C_\infty$ the connected component of the point $R(1,0)$ in
the projective line\/ $\PP(R)$ over a ring $R$. Then the following
holds:
\begin{enumerate}\itemsep=0cm
  \item
  The group\/ $\GL_2(R)$ acts transitively on the set of connected
  components of $\PP(R)$.
  \item
  Let $t_1, t_2,\ldots,t_{n}\in R$, $n\geq 0$, and put
  \begin{equation}\label{matrix-kette}
 (x,y):=(1,0)\cdot E(t_{n})\cdot E(t_{n-1})
                  \cdots E(t_1).
\end{equation}
Then $R(x,y)\in C_\infty$ and, conversely, each point $r\in C_\infty$
can be written in this way.
  \item
  The stabilizer of $C_\infty$ in\/ $\GL_2(R)$ is the group\/ $\GE_2(R)$.
  \item
  The projective line\/ $\PP(R)$ is connected if, and only if, $R$
is a\/ $\GE_2$-ring.
\end{enumerate}
\end{thm}
\begin{proof}
(a) This is immediate from the fact that the group $\GL_2(R)$ acts
transitively on the point set $\PP(R)$ and preserves the relation
$\dis$.

(b) Every matrix $E(t_i)$ appearing in (\ref{matrix-kette}) maps
$C_\infty$ onto $C_\infty$, since $R(0,1)\in C_\infty$ goes over to
$R(1,0)\in C_\infty$. Therefore $R(x,y)\in C_\infty$.

On the other hand let $r\in C_\infty$. Then there exists a sequence
of points $p_i =R(a_i,b_i)\in\PP(R)$, $i\in\{0,1,\ldots,n\}$, such
that
\begin{equation}\label{dist-kette}
R(1,0)=p_0\,\dis\,p_1\,\dis\ldots\dis\, p_n=r.
\end{equation}
Now the arbitrarily chosen admissible pairs $(a_i,b_i)$ are
``normalized'' recursively as follows: First define
$(x_{-1},y_{-1}):=(0,-1)$ and $(x_0,y_0):=(1,0)$. So
$p_0=R(x_0,y_0)$. Next assume that we already are given admissible
pairs $(x_j,y_j)$ with $p_j=R(x_j,y_j)$ for all
$j\in\{0,1,\ldots,i-1\}$, $1\leq i\leq n$. From Lemma
\ref{le:matrizen}, there are $s_{i}\in R$ and $u_{i}\in R^*$ such
that
\begin{equation}
 \Mat2{x_{i-1}&y_{i-1}\\a_{i}&b_{i}} =
 \Mat2{1&0\\s_{i}&u_{i}}\Mat2{x_{i-1}&y_{i-1}\\-x_{i-2}&-y_{i-2}}.
\end{equation}
By putting $x_{i}:=u_{i}^{-1}a_{i}$, $y_{i}:=u_{i}^{-1}b_{i}$, and
$t_{i}:=u_{i}^{-1}s_{i}$ we get
\begin{equation}\label{rekursion}
 \Mat2{x_{i}&y_{i} \\ -x_{i-1}&-y_{i-1}} =
  E(t_{i})\cdot \Mat2{x_{i-1}&y_{i-1}\\ -x_{i-2}&-y_{i-2}}
\end{equation}
and $p_{i}=R(x_{i},y_{i})$. Therefore, finally, $(x_{n},y_{n})$ is
the first row of the matrix
\begin{equation}\label{produktmatrix}
G':= E(t_n)\cdot E(t_{n-1})\cdots E(t_{1})
 \in\E_2(R),
\end{equation}
and $r=R(x_n,y_n)$.

(c) As has been noticed at the end of Section \ref{se:prelim}, the
set of all matrices (\ref{erzeugende}) generates $\E_2(R)$. This
together with (b) implies that $\E_2(R)$ stabilizes $C_\infty$.
Further, $R(1,0)$ remains fixed under each invertible diagonal
matrix. Therefore $\GE_2(R)$ is contained in the stabilizer of
$C_\infty$.

Conversely, suppose that $G\in\GL_2(R)$ stabilizes $C_\infty$. Then
the first row of $G$, say $(a,b)$, determines a point of $C_\infty$.
By (\ref{matrix-kette}) and (\ref{produktmatrix}), there is a matrix
$G'\in \E_2(R)$ and a unit $u\in R^*$ such that $(a,b)=(1,0)\cdot
(uG')$. Now Lemma \ref{le:matrizen} can be applied to $G$ and
$uG'\in\GE_2(R)$ in order to establish that $G\in\GE_2(R)$.

(d) This follows from (a) and (c).
\end{proof}

From Theorem \ref{th:komponenten} and (\ref{produktmatrix}), the
connected component of $R(1,0)\in\PP(R)$ is given by all pairs of
$(1,0)\cdot\E_2(R)$ or, equivalently, by all pairs of
$(1,0)\cdot\GE_2(R)$. Each product (\ref{matrix-kette}) gives rise to
a sequence
\begin{equation}
  (x_i,y_i)=(1,0)\cdot E(t_i)\cdot E(t_{i-1})\cdots E(t_1),\;
  i\in\{0,1,\ldots,n\},
\end{equation}
which in turn defines a sequence $p_i:=R(x_i,y_i)$ of points with
$p_0=R(1,0)$. By putting $p_n=:r$ and by reversing the arguments in
the proof of (b), it follows that (\ref{dist-kette}) is true. So, if
the diameter of $C_\infty$ is finite, say $m\geq 0$, then in order to
reach all points of $C_\infty$ it is sufficient that $n$ ranges from
$0$ to $m$ in formula (\ref{matrix-kette}).

By the action of $\GL_2(R)$, the connected component $C_p$ of any
point $p\in\PP(R)$ is $\GL_2(R)$-equivalent to the connected
component $C_\infty$ of $R(1,0)$ and the stabilizer of $C_p$ in
$\GL_2(R)$ is conjugate to $\GE_2(R)$. Observe that in general
$\GE_2(R)$ is not normal in $\GL_2(R)$. Cf. the example in
\ref{beispiele2} (c). All connected components are isomorphic
subgraphs of the distant graph.

\section{Generalized Chain Geometries}

If $K\subset R$ is a (not necessarily commutative) subfield, then the
$K$-sublines of $\PP(R)$ give rise to a {\em generalized chain
geometry\/} $\Sigma(K,R)$; see \cite{blu+h-99a}. In contrast to an
ordinary chain geometry (cf.\ \cite{herz-95}) it is not assumed that
$K$ is in the centre of $R$. Any three mutually distant points are on
at least one $K$-chain. Two distinct points are distant exactly if
they are on a common $K$-chain. Therefore each $K$-chain is contained
in a unique connected component. Each connected component $C$
together with the set of $K$-chains entirely contained in it defines
an incidence structure $\Sigma(C)$. It is straightforward to show
that the automorphism group of the incidence structure $\Sigma(K,R)$
is isomorphic to the wreath product of $\Aut\,\Sigma(C)$ with the
symmetric group on the set of all connected components of $\PP(R)$.

If $\Sigma(K,R)$ is a chain geometry then the connected components
are exactly the {\em maximal connected subspaces\/} of $\Sigma(K,R)$
\cite[793, 821]{herz-95}. Cf.\ also \cite{kroll-91} and
\cite{kroll-92}.

An $R$-semilinear bijection of $R^2$ induces an automorphism of
$\Sigma(K,R)$ if, and only if, the accompanying automorphism of $R$
takes $K$ to $u^{-1}Ku$ for some $u\in R^*$. On the other hand, if
$\PP(R)$ is disconnected then we cannot expect all automorphisms of
$\Sigma(K,R)$ to be semilinearly induced. More precisely, we have the
following:

\begin{thm}
Let $\Sigma(K,R)$ be a {\em disconnected\/} generalized chain
geometry, i.e., the projective line $\PP(R)$ over $R$ is
disconnected. Then $\Sigma(K,R)$ admits automorphisms that cannot be
induced by any semilinear bijection of $R^2$.
\end{thm}
\begin{proof}
(a) Suppose that two semilinearly induced bijections
$\gamma_1,\gamma_2$ of $\PP(R)$ coincide on all points of one
connected component $C$ of $\PP(R)$. We claim that
$\gamma_1=\gamma_2$. For a proof choose two distant points $R(a,b)$
and $R(c,d)$ in $C$. Also, write $\alpha$ for that projectivity which
is given by the matrix $\SMat2{a&b\\c&d}$. Then
$\beta:=\alpha\gamma_1\gamma_2^{-1}\alpha^{-1}$ is a semilinearly
induced bijection of $\PP(R)$ fixing the connected component
$C_\infty$ of $R(1,0)$ pointwise. Hence $R(1,0)$, $R(0,1)$, and
$R(1,1)$ are invariant under $\beta$, and we get
\begin{displaymath}
   R(x,y)^\beta = R(x^\zeta u,y^\zeta u) \mbox{ for all }(x,y)\in R^2
\end{displaymath}
with $\zeta\in\Aut(R)$ and $u\in R^*$, say. For all $x\in R$ the
point $R(x,1)$ is distant from $R(1,0)$; so it remains fixed under
$\beta$. Therefore $x=u^{-1}x^\zeta u$ or, equivalently, $x^\zeta
u=ux$ for all $x\in R$. Finally, $R(x,y)^\beta = R(u x,u y)=R(x,y)$
for all $(x,y)\in R^2$, whence $\gamma_1=\gamma_2$.

(b) Let $\gamma$ be a non-identical projectivity of $\PP(R)$ given by
a matrix $G\in \GE_2(R)$, for example, $G=B_{12}(1)$. From Theorem
\ref{th:komponenten}, the connected component $C_\infty$ of $R(1,0)$
is invariant under $\gamma$. Then
\begin{equation}\label{def-delta}
  \delta:\PP(R)\to\PP(R):
  \left\{
  \begin{array}{ll}
    p\mapsto p^\gamma &\mbox{for all } p\in C_\infty\\
    p\mapsto p        &\mbox{for all } p\in \PP(R)\setminus C_\infty
  \end{array}
  \right.
\end{equation}
is an automorphism of $\Sigma(K,R)$. The projectivity $\gamma$ and
the identity on $\PP(R)$ are different and both are linearly induced.
The mapping $\delta$ coincides with $\gamma$ on $C_\infty$ and with
the identity on every other connected component. There are at least
two distinct connected components of $\PP(R)$. Hence it follows from
(a) that $\delta$ cannot be semilinearly induced.
\end{proof}
If a cross-ratio in $\PP(R)$ is defined according to
\cite[1.3.5]{herz-95}\ then four points with cross-ratio are
necessarily in a common connected component. Therefore, the
automorphism $\delta$ defined in (\ref{def-delta}) preserves all
cross-ratios. However, cross-ratios are not invariant under $\delta$
if one adopts the definition in \cite[90]{benz-73} or
\cite[7.1]{herz-95} which works for commutative rings only. This is
due to the fact that here four points with cross-ratio can be in two
distinct connected components.

We shall give examples of disconnected (generalized) chain geometries
in the next section.

\section{Examples}

There is a widespread literature on (non-)$\GE_2$-rings. We refer to
\cite{abra-86}, \cite{huah-84}, \cite{cohn-66},  \cite{costa-88},
\cite{dennis-75}, \cite{hahn+om-89}, and  \cite{silv-81}. We are
particularly interested in rings containing a field and the
corresponding generalized chain geometries.

\begin{rem}
Let $R$ be a ring. Then each admissible pair $(x,y)\in R^2$ is {\em
unimodular}, i.e., there exist $x',y'\in R$ with $xx'+yy'=1$. We
remark that
\begin{equation}\label{uni-admiss}
  (x,y)\in R^2 \mbox{ unimodular} \Rightarrow (x,y)\mbox{ admissible}
\end{equation}
is satisfied, in particular, for all {\em commutative\/} rings, since
$xx'+yy'=1$ can be interpreted as the determinant of an invertible
matrix with first row $(x,y)$. Also, all rings of {\em stable rank\/}
$2$ \cite[293]{veld-85} satisfy (\ref{uni-admiss}); cf.\
\cite[2.11]{veld-85}. For example, local rings, matrix rings over
fields, and finite-dimensional algebras over commutative fields are
of stable rank $2$. See \cite[4.1B]{hahn+om-89}, \cite[\S
2]{veld-85}, \cite{veld-95}, and the references given there.

The following example shows that (\ref{uni-admiss}) does not hold for
all rings: Let $R:=K[X,Y]$ be the polynomial ring over a proper skew
field $K$ in independent central indeterminates $X$ and $Y$. There
are $a,b\in K$ with $c:=ab-ba\neq 0$. From
\begin{displaymath}
  (X+a)(Y+b)c^{-1}-(Y+b)(X+a)c^{-1}=1,
\end{displaymath}
the pair $(X+a,-(Y+b))\in R^2$ is unimodular. However, this pair is
not admissible: Assume to the contrary that $(X+a,-(Y+b))$ is the
first row of a matrix $M\in\GL_2(R)$ and suppose that the second
column of $M^{-1}$ is the transpose of $(v_0,w_0)\in R^2$. Then
\begin{displaymath}
  P:= \{(v,w)\in R^2 \mid (X+a)v-(Y+b)w=0\} = (v_0,w_0)R.
\end{displaymath}
On the other hand, by \cite[Proposition 1]{ojan+s-71}, the right
$R$-module $P$ cannot be generated by a single element, which is a
contradiction.
\end{rem}

\begin{exas}\label{beispiele1}
\begin{enumerate}
\item
If $R$ is a ring of stable rank $2$  then $\PP(R)$ is connected and
its diameter is $\leq 2$ \cite[Proposition 1.4.2]{herz-95}. In
particular, the diameter is $1$ exactly if $R$ is a field and it is
$0$ exactly if $R=\{0\}$.

As has been pointed out in \cite[(2.1)]{bart-89}, the points of the
projective line over a ring $R$ of stable rank $2$ are exactly the
submodules $R(t_2 t_1+1,t_2)$ of $R^2$ with $t_1,t_2\in R$. Clearly,
this is just a particular case of our more general result in Theorem
\ref{th:komponenten} (b).

Conversely, if an arbitrary ring $R$ satisfies (\ref{uni-admiss}) and
$\PP(R)$ is connected with diameter $\leq 2$, then $R$ is a ring of
stable rank $2$ \cite[Proposition 1.1.3]{herz-95}.

\item
The projective line over a (not necessarily commutative) Euclidean
ring $R$ is connected, since every Euclidean ring is a $\GE_2$-ring
\cite[Theorem 1.2.10]{hahn+om-89}.

\end{enumerate}
\end{exas}

Our next examples are given in the following theorem:

\begin{thm}\label{th:unendlichdim}
Let $U$ be an infinite-dimensional vector space over a field $K$ and
put $R:=\End_K(U)$. Then the projective line $\PP(R)$ over $R$ is
connected and has diameter $3$.
\end{thm}
\begin{proof}
We put $V:=U\times U$ and denote by $\cG$ those subspaces $W$ of $V$
that are isomorphic to $V/W$. By \cite[2.4]{blunck-99}, the mapping
\begin{equation}
  \Phi : \PP(R)\to \cG :
  R(\alpha,\beta)\mapsto\{(u^\alpha,u^\beta)\mid u\in U \}
\end{equation}
is bijective and two points of $\PP(R)$ are distant exactly if their
$\Phi$-images are complementary. By an abuse of notation, we shall
write $\dist(W_1,W_2)=n$, whenever $W_1,W_2$ are $\Phi$-images of
points at distance $n$, and $W_1\,\dis\,W_2$ to denote complementary
elements of $\cG$. As $V$ is infinite-dimensional, $2\dim W=\dim
V=\dim W$ for all $W\in\cG$.

We are going to verify the theorem in terms of $\cG$: So let
$W_1,W_2\in\cG$. Put $Y_{12}:=W_1\cap W_2$ and choose $Y_{23}\leq
W_2$ such that $W_2=Y_{12}\oplus Y_{23}$. Then $W_1\cap Y_{23}=\{0\}$
so that there is a $W_3\in\cG$ through $Y_{23}$ with $W_1\,\dis\,
W_3$. By the law of modularity,
\begin{displaymath}
  W_2\cap W_3 = (Y_{23}+Y_{12})\cap W_3 = Y_{23}+(Y_{12}\cap W_3)=Y_{23}.
\end{displaymath}
Finally, choose $Y_{14}\leq W_1$ with $W_1=Y_{12}\oplus Y_{14}$ and
$Y_{34}\leq W_3$ with $W_3=Y_{23}\oplus Y_{34}$. Hence we arrive at
the decomposition
\begin{equation}\label{dirsumme}
  V=Y_{14}\oplus Y_{12}\oplus Y_{23}\oplus Y_{34}.
\end{equation}
As $W_2\in\cG$, so is also $W_4:=Y_{14}\oplus Y_{34}$. Now there are
two possibilities:

Case 1: There exists a linear bijection $\sigma:Y_{14}\to Y_{23}$. We
define $Y:=\{v+v^\sigma\mid v\in Y_{14}\}$. Then $Y_{14}$, $Y_{23}$,
and $Y$ are easily seen to be mutually complementary subspaces of
$Y_{14}\oplus Y_{23}$. Therefore, from (\ref{dirsumme}),
\begin{equation}
  V=Y_{14}\oplus Y_{12}\oplus Y     \oplus Y_{34}
   =Y     \oplus Y_{12}\oplus Y_{23}\oplus Y_{34},
\end{equation}
i.e., $W_1\,\dis\,(Y\oplus Y_{34})\,\dis\,W_2$. So
$\dist(W_1,W_2)\leq 2$.

Case 2: $Y_{14}$ and $Y_{23}$ are not isomorphic. Then $\dim
Y_{12}=\dim W_1$, since otherwise $\dim Y_{12}<\dim W_1=\dim W_2$
together with well-known rules for the addition of infinite cardinal
numbers would imply
\begin{eqnarray}
  &&\dim W_1 = \max\{\dim Y_{12},\dim Y_{14}\} = \dim Y_{14},\nonumber\\
  &&\dim W_2 = \max\{\dim Y_{12},\dim Y_{23}\} = \dim Y_{23},\nonumber
\end{eqnarray}
a contradiction to $\dim Y_{14}\neq \dim Y_{23}$.

Likewise, it follows that $\dim Y_{34} =\dim W_3$. But this means
that $Y_{12}$ and $Y_{34}$ are isomorphic, whence the proof in case 1
can be modified accordingly to obtain a $Y\leq Y_{12}\oplus Y_{34}$
such that $W_1\,\dis\,W_3\,\dis\,(Y\oplus Y_{14})\,\dis\, W_2$. So
now $\dist(W_1,W_2)\leq 3$.

It remains to establish that in $\cG$ there are elements with
distance $3$: Choose any subspace $W_1\in\cG$ and a subspace $W_2\leq
W_1$ such that $W_1/W_2$ is $1$-dimensional. With the previously
introduced notations we get $Y_{12}=W_2$, $\dim Y_{14}=1$,
$Y_{23}=\{0\}$, $Y_{34}=W_3\in\cG$, and $W_4=Y_{14}\oplus W_3$. As
before, $V=W_2\oplus W_4$ and from $\dim W_2 = 1+\dim W_2=\dim W_1
=\dim W_3=1+\dim W_3 = \dim W_4$ we obtain $W_2,W_4\in\cG$. By
construction, $\dist (W_1,W_2)\neq 0,1$. Also, this distance cannot
be $2$, since $W\,\dis\,W_1$ implies $W+W_2\neq V$ for all $W\in\cG$.

This completes the proof.
\end{proof}

If $K$ is a proper skew field, then $K$ can be embedded in
$\End_K(U)$ in several ways  \cite[17]{blunck-00a}; each embedding
gives rise to a connected generalized chain geometry. (In
\cite{blunck-00a} this is just called a ``chain geometry''.) If $K$
is commutative, then $\End_K(U)$ is a $K$-algebra and $x \mapsto
x\,\id_U$ is a distinguished embedding of $K$ into the centre of
$\End_K(U)$. In this way an ordinary connected chain geometry arises;
cf.\ \cite[4.5.\ Example (4)]{herz-95}.

Our next goal is to show the existence of chain geometries with
connected components of infinite diameter.
\begin{rem}\label{eindeutigkeit}
If $R$ is an arbitrary ring then each matrix $A\in\GE_2(R)$ can be
expressed in {\em standard form\/}
\begin{equation}\label{standard-form}
  A=\diag(u,v)\cdot E(t_n)\cdot E(t_{n-1})\cdots E(t_1),
\end{equation}
where $u,v\in R^*$,  $t_1,t_n\in R$, $t_2,t_3,\ldots,t_{n-1}\in
R\setminus (R^*\cup\{0\})$, and $t_1,t_2\neq 0$ in case $n=2$
\cite[Theorem (2.2)]{cohn-66}. Since $E(0)^2=\diag(-1,-1)$, each
matrix $A\in\GE_2(R)$ can also be written in the form
(\ref{standard-form}) subject to the slightly modified conditions
$u,v\in R^*$, $t_1,t_n\in R$, $t_2,t_3,\ldots,t_{n-1}\in R\setminus
(R^*\cup\{0\})$, and $n\geq 1$. We call this a {\em modified standard
form\/} of $A$.

Suppose that there is a unique standard form for $\GE_2(R)$. For all
non-diagonal matrices in $\GE_2(R)$ the unique representation in
standard form is at the same time the unique representation in
modified standard form. Any diagonal matrix $A\in\GE_2(R)$ is already
expressed in standard form, but its unique modified standard form
reads $A=-A\cdot E(0)^2$. Therefore there is also a unique modified
standard form for $\GE_2(R)$.

By reversing these arguments it follows that the existence of a
unique modified standard form for $\GE_2(R)$ is equivalent to the
existence of a unique standard form for $\GE_2(R)$.
\end{rem}

\begin{thm}\label{th:dm-unendlich}
Let $R$ be a ring with a unique standard form for $\GE_2(R)$ and
suppose that $R$ is not a field. Then every connected component of
the projective line $\PP(R)$ over $R$ has infinite diameter.
\end{thm}
\begin{proof}
Since $R$ is not a field, there exists an element $t\in R\setminus
(R^*\cup\{0\})$. We put
\begin{equation}\label{def.q_m}
  q_m:=R(c_m,d_m) \mbox{ where } (c_m,d_m):=(1,0)\cdot E(t)^m
  \mbox{ for all } m\in\{0,1,\ldots\}.
\end{equation}
Next fix one $m\geq 1$, and put $n-1:=\dist(q_0,q_{m-1})\geq 0$.
Hence there exists a sequence
\begin{equation}
  p_0\,\dis\,p_1\,\dis\ldots\dis\,p_{n-1}\,\dis\,p_{n}
\end{equation}
such that $p_0=q_0$, $p_{n-1}=q_{m-1}$, and $p_n=q_m$. Now we proceed
as in the proof of Theorem \ref{th:komponenten} (b): First let
$p_i=R(a_i,b_i)$ and put $(x_{-1},y_{-1}):=(0,-1)$,
$(x_0,y_0):=(1,0)$. Then pairs $(x_i,y_i)\in R^2$ and matrices
$E(t_i)\in\E_2(R)$ are defined in such a way that $p_i=R(x_i,y_i)$
and that (\ref{rekursion}) holds for $i\in \{1,2\ldots,n\}$. It is
immediate from (\ref{rekursion}) that a point $p_{i}$, $i\geq 2$, is
distant from $p_{i-2}$ exactly if $t_{i}\in R^*$. Also, $p_{i}=
p_{i-2}$ holds if, and only if, $t_{i}=0$. We infer from
(\ref{rekursion}) and $\dist(p_i,p_{j})= {\mid i-j\mid}$ for all
$i,j\in\{0,1,\ldots,n-1\}$ that
\begin{equation}\label{erste.darstellung}
  \Mat2{x_{n}&y_{n}\\-x_{n-1}&-y_{n-1}}=
   E(t_{n})\cdot E(t_{n-1})\cdots E(t_1),
\end{equation}
where $t_i\in R\setminus (R^*\cup\{0\})$ for all
$i\in\{2,3,\ldots,n-1\}$. On the other hand, by (\ref{def.q_m}) and
$(c_{m-1},d_{m-1})=(0,-1)\cdot E(t)^m$, there are $v,v'\in R^*$ with
\begin{equation}\label{zweite.darstellung}
  \Mat2{x_{n}&y_{n}\\-x_{n-1}&-y_{n-1}}=\diag(v,v')\cdot
  E(t)^{m}.
\end{equation}
From Remark \ref{eindeutigkeit}, the modified standard forms
(\ref{erste.darstellung}) and (\ref{zweite.darstellung}) are
identical. Therefore, $n=m$, $\dist(q_0,q_{m-1})=m-1$, and the
diameter of the connected component of $q_0$ is infinite.

By Theorem \ref{th:komponenten} (a), all connected components of
$\PP(R)$ have infinite diameter.
\end{proof}

\begin{rem}\label{K-ring}
Let $R$ be a ring such that $R^*\cup\{0\}$ is a field, say $K$, and
suppose that we have a {\em degree function}, i.e.\ a function
$\deg:R\to\{-\infty\}\cup\{0,1,\ldots\}$ satisfying
\begin{eqnarray*}
 \deg a &=& -\infty\quad \mbox{ if, and only if, } a=0,\\
 \deg a &=& 0 \quad \mbox{ if, and only if, } a\in R^*,\\
 \deg (a+b)&\leq& \max\{\deg a,\deg b\},\\
 \deg (ab) &=& \deg(a)+\deg(b),
\end{eqnarray*}
for all $a,b\in R$. Then, following \cite[21]{cohn-66}, $R$ is called
a {\em $K$-ring with a degree function}.

If $R$ is a $K$-ring with a degree function, then there is a unique
standard form for $\GE_2(R)$ \cite[Theorem (7.1)]{cohn-66}.
\end{rem}

\begin{exas}\label{beispiele2}
\begin{enumerate}
\item
Let $R$ be a $K$-ring with a degree-function such that $R\neq K$.
From Remark \ref{K-ring} and Theorem \ref{th:dm-unendlich}, all
connected components of the projective line $\PP(R)$ have infinite
diameter.

The associated generalized chain geometry $\Sigma(K,R)$ has a lot of
strange properties. For example, {\em any two\/} distant points are
joined by a unique $K$-chain. However, we do not enter a detailed
discussion here.
\item
Let $K[X]$ be the polynomial ring over a field $K$ in a central
indeterminate $X$. From (a) and Example \ref{beispiele1} (b), the
projective line $\PP(K[X])$ is connected and its diameter is
infinite. On the other hand, if $K$ is commutative then $K[X]$ has
stable rank $3$ \cite[2.9]{veld-95}; see also \cite[Chapter V,
(3.5)]{bass-68}. So there does not seem to be an immediate connection
between stable rank and diameter.
\item
Let $R:=K[X_1,X_2,\ldots,X_m]$ be the polynomial ring over a field
$K$ in $m>1$ independent central indeterminates. Then, by an easy
induction and by \cite[Proposition (7.3)]{cohn-66},
\begin{equation}\label{xy-matrix}
   A_n:=\Mat2{1+X_1X_2&X_1^2\\-X_2^2&1-X_1X_2}^n =
  \Mat2{1+nX_1X_2&nX_1^2\\-nX_2^2&1-nX_1X_2}
\end{equation}
is in $\GL_2(R)\setminus\GE_2(R)$ for all $n\in\ZZ$ that are not
divisible by the characteristic of $K$. Also, the inner automorphism
of $\GL_2(R)$ arising from the matrix $A_1$ takes
$B_{12}(1)\in\E_2(R)$ to a matrix that is not even in $\GE_2(R)$; see
\cite[121--122]{silv-81}. So neither $\E_2(R)$ nor $\GE_2(R)$ is a
normal subgroup of $\GL_2(R)$.

We infer that the projective line over $R$ is not connected. Further,
it follows from (\ref{xy-matrix}) that the number of right cosets of
$\GE_2(R)$ in $\GL_2(R)$ is infinite, if the characteristic of $K$ is
zero, and $\geq\Char K$ otherwise. From Theorem \ref{th:komponenten},
this number of right cosets is at the same time the number of
connected components in $\PP(R)$. Even in case of $\Char K=2$ there
are at least three connected components, since the index of
$\GE_2(R)$ in $\GL_2(R)$ cannot be two. From (a), all connected
components of $\PP(R)$ have infinite diameter.

So, for each commutative field $K$, we obtain a disconnected chain
geometry $\Sigma(K,R)$, whereas for each skew field $K$ a
disconnected generalized chain geometry arises.
\end{enumerate}
\end{exas}

Acknowledgement. The authors are obliged to the referee for pointing
out the paper \cite{ojan+s-71}.


Institut f\"ur Geometrie\\ Technische Universit\"at\\ Wiedner
Hauptstra{\ss}e 8--10\\ A--1040 Wien\\Austria
\end{document}